\documentclass[11pt]{amsart}
\usepackage{amssymb,graphics,epsfig,hyperref,overpic}

\newtheorem{thm}{Theorem}[section]
\newtheorem{lem}[thm]{Lemma}

\newtheorem{prop}[thm]{Proposition}

\theoremstyle{definition}

\newtheorem{note}[thm]{Note}

\theoremstyle{remark}

\newcommand{\R}{\mathbf{R}}
\newcommand{\RP}{\mathbf{RP}}

\newcommand{\ol}[1]{{\overline #1}}

\newcommand{\K}{\mathcal{K}}
\newcommand{\B}{\mathcal{B}}
\newcommand{\C}{\mathcal{C}}

\renewcommand{\d}{\partial}
\renewcommand{\P}{\mathcal{P}}
\renewcommand{\H}{\mathcal{H}}
\renewcommand{\S}{\mathbf{S}}
\renewcommand{\l}{\langle}
\renewcommand{\r}{\rangle}

\DeclareMathOperator{\inte}{int}

\DeclareMathOperator{\rc}{rc}
\DeclareMathOperator{\bc}{bc}
\DeclareMathOperator{\nc}{nc}
\DeclareMathOperator{\cm}{cm}
\DeclareMathOperator{\cl}{cl}
\DeclareMathOperator{\cd}{cd}
\DeclareMathOperator{\relbd}{rbd}
\DeclareMathOperator{\relint}{ri}
\DeclareMathOperator{\grad}{grad}
\DeclareMathOperator{\aff}{aff}
\DeclareMathOperator{\apex}{apex}

\begin{document}

\vspace*{-0.5in}

\title[Unbounded Convex Bodies and Hypersurfaces]{Deformations  of Unbounded Convex Bodies\\ and Hypersurfaces}

\author{Mohammad Ghomi}
\address{School of Mathematics, Georgia Institute of Technology,
Atlanta, GA 30332}
\email{ghomi@math.gatech.edu}
\urladdr{www.math.gatech.edu/$\sim$ghomi}
\date{Last Typeset \today.}
\thanks{Supported by NSF Grant DMS-0336455.}

\begin{abstract}
We study the topology of the space $\d\K^n$ of  complete convex hypersurfaces of $\R^n$  which are homeomorphic to $\R^{n-1}$.
In particular, using Minkowski sums, we construct  a  deformation retraction  of $\d\K^n$ onto the Grassmannian space of hyperplanes.  So every hypersurface  in $\d \K^n$ may be flattened in a canonical way. Further, the total curvature of each hypersurface evolves continuously and monotonically under this deformation. 
 We also show that, modulo proper rotations,  the subspaces of $\d\K^n$ consisting of smooth, strictly convex, or positively curved hypersurfaces are each contractible, which settles a question of H. Rosenberg. 
\end{abstract}

\maketitle


\section{Introduction}
It is easy to see  that the space  of compact convex bodies $K\subset\R^n$, and their boundary hypersurfaces $\d K$, are contractible under Hausdorff topology. Indeed,  the Minkowski addition and scalar multiplication  yields a canonical homotopy 
\begin{equation}\label{eq:first}
K_t:=(1-t)K+t\, \mathbf{B}^n
\end{equation}  
between  $K$ and the unit ball $\mathbf{B}^n$, while $\d K_t$ deforms $\d K$ to  the sphere $\S^{n-1}=\d \mathbf{B}^n$.  Here we construct analogous   deformations for  \emph{unbounded convex bodies}, i.e., closed noncompact convex subsets of $\R^n$ with interior points. The most significant class of these objects is the space
$\K^{n}$  of convex bodies  with $\d K$ homeomorphic to $\R^{n-1}$, since  any other  convex body is the sum of a compact  convex set with a linear space (Lemma \ref{lem:rc4}). We study $\K^n$ (and the  corresponding space of boundary hypersurfaces $\d \K^n$) with respect to a refinement of the bounded-Hausdorff topology, called \emph{asymptotic topology}  (Section \ref{subsec:top}), which  ensures the continuity of  the  \emph{total curvature} function $\tau\colon \K^n\to\R$ (Proposition \ref{lem:olnc}). Here $\tau(K)$ is the measure in $\S^{n-1}$ of the \emph{unit normal cone}, or outward unit normals to support hyperplanes of $K$. Also let
$\H^n$ be the collection of  half-spaces in $\R^n$ whose boundaries pass through the origin to form the Grassmannian space $\d \H^n=Gr(n-1,n)$.

\begin{thm}\label{thm:main}
$\K^n$ (resp. $\d \K^n$) admits a regularity preserving strong deformation retraction onto  $\H^n$  (resp. $\d \H^n$) with respect to the asymptotic topology. Under this deformation the total curvature of each element of $\K^n$ (resp. $\d \K^n$) evolves monotonically, as its unit normal cone uniformly shrinks   to a single vector. 
\end{thm}

The term \emph{strong deformation retraction}  here means that there exists a  continuous map $\K^n\times [0,1]\to \K^n$, $(K,t)\mapsto K_t$, such that $K_0=K$, $K_1\in \H^n$, and $K_t$ is constant for all  $K\in\H^n$. Further, by \emph{regularity preserving} we mean that if $\d K_0$  is of regularity  class $C^k$, for $1\leq k\leq\infty$ or $k=\omega$, then  $\d K_t$ is also (at least) $C^k$ for all $t$.
Note  that $\H^n$ and $\d \H^n$ are homeomorphic to $\S^{n-1}$ and the real projective space $\RP^{n-1}$ respectively; thus Theorem \ref{thm:main} shows that $\K^n$ and $\d\K^n$ are not topologically trivial. On the other hand, since the special orthogonal group $SO(n)$ acts transitively on $\H^n$ and $\d \H^n$, it follows that the quotient spaces $\K^n/SO(n)$ and $\d \K^n/SO(n)$ are contractible. A similar phenomenon also holds for certain subspaces of $\K^n$ and $\d \K^n$ by the next result. Here $\K^n_+$ denotes the space of unbounded convex bodies which are strictly convex at some point (Section \ref{subsec:notation}), and $\P^n$ is  the space of paraboloidal convex bodies  generated by the action of $SO(n)$ on   $\{x_n\geq\sum_{i=1}^{n-1} x_i^2\}$. Further,  $\d \K^n_+$ and $\d \P^n$ are the spaces of the corresponding boundary hypersurfaces.

\begin{thm}\label{thm:main2}
$\K^n_+$ (resp. $\d \K^n_+$) admits a  regularity preserving strong deformation retraction onto $\P^n$ (resp. $\d \P^n$) with respect to the asymptotic topology. Furthermore, if an element of $\K^n_+$ (or $\d \K^n_+$) is  strictly convex, or  has positive Gaussian curvature, then each of these properties will be preserved under the deformation. 
\end{thm}

The first step in proving the above theorems  is to partition  $\K^n$  into subsets each of which is associated with a certain unit vector $u\in\S^{n-1}$ called the \emph{central direction}. This notion, which refines some previous works of Wu \cite{wu:spherical}, will be developed  in Section \ref{sec:crd}. Then we show in Section \ref{sec:proofs} that the bodies  in each direction class may be deformed to a preferred body within that class (i.e., a half-space or a solid paraboloid). 
Similar to  \eqref{eq:first}, these deformations are constructed explicitly using Minkowski sums; however, the situation here is considerably more involved since the sum of a pair of convex bodies in $\K^n$  or $\K^n_+$ may no longer belong to these spaces; not to mention that the Minkowski addition is not even a continuous operation on $\K^n$.  Thus in Sections \ref{sec:continuous} and \ref{sec:sum} we derive the conditions  for the Minkowski addition to preserve the natural  geometric and topological  properties we need, which may also be of independent interest.   

 Unbounded convex bodies arise naturally in convex analysis and optimization  since they form the epigraphs of convex functions \cite{rockafellar, rw}; while unbounded convex hypersurfaces have been  studied in differential geometry in terms of their local characterizations \cite{stoker,vH:convex,sacksteder:convex}, and isometric embeddings \cite{alexandrov:polyhedra, pogorelov:book}. See also \cite{alexander&ghomi:chp,alexander&ghomi:chpII,agw} for more recent results involving  hypersurfaces with boundary. The main motivation for this work, however, arises from the study of regular homotopy classes of positively curved surfaces in $\R^3$ by Gluck and Pan \cite{gluck&pan}, which was  generalized to higher dimensions by the author and Kossowski \cite{ghomi&kossowski} via the $h$-principle \cite{gromov:pdr, eliashberg&mishachev}. These papers described  the path components of the space of compact positively curved hypersurfaces with boundary; thus setting the stage for exploring the topology  of the space of \emph{complete} positively curved hypersurfaces in this paper.  The study of  homotopy with curvature bounds originates with the works of Feldman \cite{feldman:curve, feldman:mean}, and has  been a subject of interest since then  \cite{little:curve, little:torsion, gluck&pan, ghomi&kossowski,ghomi:knots,ekholm1,ekholm2,rosenberg:constant}.

\begin{note}
It might be tempting to think that the flattening procedure of Theorem \ref{thm:main} could be carried out also by means of the (reverse) mean curvature flow, at least in the smooth case; however, there are unbounded convex hypersurfaces such as  the ``grim-reaper"  in $\R^2$ \cite{angenent2}, given by $y=\log(\cos(x))$, which evolve by translations and thus never become flat. 
Such self-similar solutions exist also in higher dimensions \cite{eh}, and for Gauss curvature flow \cite{urbas}, which highlight the utility  of the algebraic approach to surface deformation studied here.
\end{note}

\section{Preliminaries: Finding a Suitable Topology}\label{sec:top}
 The standard topology on the space of compact convex bodies is the Hausdorff metric topology, which admits a direct generalization to the space of unbounded bodies; however, this topology is much too rigid at infinity to allow  the deformations  we seek. On the other hand, the most common relaxation of  the Hausdorff topology,  which is known as bounded-Hausdorff topology, is too weak for our purposes here since it does not force the continuity of the total  curvature. To control the total curvature, we must control the recession cones, and  we strengthen the bounded-Hausdorff topology accordingly, as described below.

\subsection{Basic notation and terminology}\label{subsec:notation}
In this paper $\R^n$ is the $n$-dimensional Euclidean space with origin $o$, standard inner product $\l\cdot,\cdot\r$ and norm $\|\cdot\|$. The sphere $\S^{n-1}$ and ball $\mathbf{B}^n$  consist of  points $x\in\R^n$ such that $\|x\|=1$ and $\|x\|\leq 1$ respectively. For  $A$, $B\subset\R^n$, the \emph{Minkowski sum} $A+B$ is  the collection of all $a+b$ where $a\in A$ and $b\in B$. Further for any $\lambda\in\R$, $\lambda A$ denotes the set  of all $\lambda a$, where $a\in A$.  A hyperplane $\d H\subset\R^n$ \emph{supports} $A$ when $\d H\cap A\neq\emptyset$ and $A$ lies on one side of $\d H$. If $A\not\subset\d H$, then the \emph{outward normal} of $\d H$ is the unit  vector normal to $\d H$ which points into the half-space determined by $\d H$ which does not contain $A$. The \emph{affine hull} of $A$,  $\aff(A)$, is  the affine subspace of least dimension  which contains $A$. The \emph{relative interior} $\relint(A)$ and \emph{relative boundary} $\relbd(A)$ are the interior and boundary of $A$ as  subsets of $\aff(A)$.  We say $A$ is \emph{relatively proper} if $A\neq\aff(A)$.
In this paper a \emph{convex body} $K\subset\R^n$ is  always a  closed convex set with interior points, and a \emph{convex hypersurface} is the boundary $\d K$ of a convex body. Our main object of study is the space $\K^n$ of  unbounded convex bodies whose boundary is homeomorphic to $\R^{n-1}$.
We also frequently mention the space $\K^n_+$  of those unbounded convex bodies $K\subset\R^n$ which have a strictly convex  point $p\in\d K$, i.e., there exists a support hyperplane $\d H$ of $K$ such that $K\cap\d H=\{p\}$. Lemmas \ref{lem:rc4} and \ref{lem:kpluscompact} below show that $\K^n_+\subset \K^n$ as we describe in the next subsection. 

\subsection{Recession cones}\label{subsec:rc}
Here we record for easy reference some basic facts on recession cones, which are instrumental in studying the asymptotic behavior of convex bodies \cite{rockafellar,rw}.
The \emph{recession cone} (a.k.a. asymptotic or horizon cone) of a convex body $K\subset\R^n$  is defined as 
$$
\rc(K):=\{\,v\in\R^n\mid K+v\subset K\,\}.
$$
 It is well known \cite[Thm. 8.2]{rockafellar} that $rc(K)$ is closed, convex, and is a \emph{cone}, i.e., for every $x\in\rc(K)$, and $\lambda> 0$, $\lambda x\in \rc(K)$. Note that $\rc(K)$ is the limit, in the sense of bounded Hausdorff topology, of rescalings $\lambda K$ as $\lambda\to 0$ ($\rc(K)$ represents what $K$ looks like when viewed from far away).
  It is easy to see that $\rc(K)$ contains a half-line $\ell$ if and only if  every point of $K$ is the source  of a half-line parallel to $\ell$. Further it can be shown that if $K$ contains a half-line $\ell$ then every point of $K$ is the source of a half-line parallel to $\ell$ \cite[Thm. 8.3]{rockafellar}. Furthermore, a closed convex set must contain a half-line if it is unbounded \cite[Thm. 8.4]{rockafellar}. In summary, we have
 
 \begin{lem}[\cite{rockafellar}]\label{lem:rc1}
 Let $K\subset\R^n$ be a convex body and $u\in\S^{n-1}$. Then the following are equivalent:
 \begin{enumerate}
 \item A point of $K$ is the source of a half-line parallel to $u$.
 \item Every point of $K$ is the source of a half-line parallel to $u$.
  \item $u\in\rc(K)$.
 \end{enumerate}
 In particular, $K$ is bounded if and only if $\rc(K)=o$.\qed
\end{lem}

Thus through every point of $K$ there passes  an affine subspace $A$ parallel to a linear subspace $L\subset\R^n$ if and only if $L\subset\rc(K)$. If $L$ has maximal dimension, then we call $L$ the \emph{linearity space} of $K$. So
 $K=\ol K+L$, where $\ol K$ is the  projection of $K$ into $L^\perp$, the orthogonal complement of $L$.
Since $\ol K$ does not contain any lines,   $\d\ol K$ is either homeomorphic to a sphere or a Euclidean space \cite[p. 3]{busemann}. (To see this observe that for any convex body $K\subset\R^n$, $\d  K$ is homeomorphic to $\S^{n-1}-\S^{n-1}\cap\rc(K)$, since if $o$ is in the interior of $K$, then every half-line originating from $o$ which is not in $\rc(K)$ intersects $\d K$ in a unique point.) Thus we have:

\begin{lem}\label{lem:rc4}
If $K\varsubsetneq\R^n$ is an unbounded convex body, then  $$K=\ol K+L,$$ where $L$ is the linearity space of $K$ and $\ol K$ is the orthogonal projection of $K$ into $L^\perp$.  Furthermore, if $\dim(L)=m$, then $\d\ol K$ is either homeomorphic to $\R^{n-m-1}$ or $\S^{n-m-1}$,  and so $\d K$ is
either homeomorphic to $\R^{n-m-1}\times\R^m=\R^{n-1}$ or $\S^{n-m-1}\times\R^m$ respectively. In particular, $K\in\K^n$ if and only if $\ol K$ is unbounded.\qed
\end{lem}
 
 So any proper unbounded convex body $K\not\in\K^n$ is the sum of a compact set with a linear space. In this sense, $\K^n$ may be regarded as the space of \emph{irreducible} proper convex bodies. Further note that if $K$ is a convex body without lines, then $K=\ol K$ in the above lemma. In particular $\ol K$ is unbounded and so $K\in \K^n$. This argument yields that $\K^n_+\subset\K^n$, since as we will show in Lemma \ref{lem:kpluscompact} below, the elements of $\K^n_+$ contain no lines.

\subsection{Asymptotic topology}\label{subsec:top}
Let us recall that  compact subsets of a metric space $M$ admit a standard topology induced by the \emph{Hausdorff distance}  defined as follows. For any set $A\subset M$, let $A_r$ denote the collection of points which are within a distance $r$ of $A$. Then the Hausdorff distance between  $A$, $B\subset M$ is given by
$$
d_h(A,B):=\inf\{r\geq 0\mid A\subset B_r,\quad\text{and}\quad B\subset A_r\}.
$$
This defines a metric on the space of compact subsets of $M$.
 The corresponding topology may be extended to the collection of closed subsets of $M$ by using the sets $A_r$ as basis elements, i.e., by stipulating that $A_i$ converges to $A$, $A_i\overset{h}{\to} A$, provided that $A_i$ eventually lies in $A_r$ for any given $r>0$.
The resulting \emph{Hausdorff topology}, however, will be too strong for our purposes here;  because under this topology the family of convex bodies in $\R^2$ given by $y\geq t |x|$ does not converge to the upper-half plane $\mathbf{H}^2$, as $t\to 0$. 

One way to weaken the Hausdorff topology so that it becomes asymptotically less rigid, is via a one-point compactification: for instance if $\pi\colon \S^{n}-\{(0,\dots,0,1)\}\to \R^{n}$ denotes the stereographic projection,  we may define the \emph{bounded-Hausdorff distance} between subsets of $\R^n$ as
\begin{equation}\label{eq:w}
d_{bh}(A,B):=d_h\Big(\pi^{-1}(A),\pi^{-1}(B)\Big).
\end{equation}
 Then the collection of the closed subsets of $\R^n$ turns into a metric space, and we call the corresponding topology the \emph{bounded-Hausdorff topology}.
 In particular note that a family of closed sets $A_i\subset\R^n$ converges to $A$ with respect to the bounded-Hausdorff topology,
$A_i\overset{bh}{\to} A$, if and only if $A_i\cap C\overset{h}{\to} A\cap C$ for every bounded set $C\subset\R^n$. 
This notion  has also been called ``set convergence" \cite[Cor. 4.7]{rw}, and agrees with  Wijsman, Attouch-Wets, and Fell topologies  \cite{beer} in finite dimensions. 

But the bounded-Hausdorff topology has a significant  shortcoming: consider the family of convex bodies in $\R^2$ given by $y\geq tx^2$. As $t\to 0$, these bodies converge to  $\mathbf{H}^2$, with respect to the bounded-Hausdorff topology, while their recession cones converge to the upper half of the $y$-axis (instead of $\rc(\mathbf{H}^2)=\mathbf{H}^2$). So under the bounded-Hausdorff topology, the mapping $K\mapsto\rc(K)$ is not continuous. This is an essential requirement, however, for  controlling the total curvature of our deformations (Proposition  \ref{lem:olnc}). So we enhance the bounded-hausdorff topology as follows:  for  convex  sets $K_0$, $K_1\subset\R^n$, define their \emph{asymptotic distance} as
$$
d_a(K_0,K_1) := d_{bh}(K_0,K_1)+d_{bh}\Big(\rc(K_0),\rc(K_1)\Big).
$$
 The topology induced by this metric on closed convex subsets of $\R^n$, which we call the \emph{asymptotic  topology}, is the one which we impose on $\K^n$. This also topologizes $\d \K^n$ via the boundary map   $K\overset{\d}{\mapsto}\d K$. In other words, a sequence of hypersurfaces $M_i\in\d \K^n$ converges asymptotically to $M_\infty\in\d \K^n$, $M_i\overset{a}{\to} M_\infty$, provided that there is a sequence of convex bodies $K_i\in \K^n$ which converges asymptotically to $K_\infty\in\K^n$, $K_i\overset{a}{\to} K_\infty$, where $\d K_i=M_i$ and  $\d K_\infty=M_\infty$.
  
\subsection{Normal cones and total curvature}\label{subsec:curvature}
The \emph{normal cone} $\nc(K)$ of a convex body $K\subset\R^n$ is the set of all outward normals of support hyperplanes of $K$ plus $o$. Wu has shown that, while $\nc(K)$ is not in general convex (!),  its  closure, $\cl(\nc(K))$ is always a closed convex cone \cite{wu:spherical}.
The \emph{unit normal cone} of $K$ is denoted by $\ol \nc(K):=\nc(K)\cap\S^{n-1}$. The Hausdorff measure of $\ol\nc(K)\subset\S^{n-1}$ is the \emph{total curvature} $\tau(K)$, which coincides with the integral of the Gaussian curvature of  $\d K$ when $\d K$ is $C^2$. The purpose of this section is to check that $\tau(K)$ is well-behaved with respect to the asymptotic topology defined above. Specifically, if we let $\C^n$ denote the space of closed convex cones in $\R^n$, then we show:

\begin{prop}\label{lem:olnc}
The mapping $\K^n\ni K\mapsto\cl(\nc(K))\in\C^n$ is continuous with respect to the asymptotic topologies on $\K^n$ and $\C^n$.
In particular, the total curvature  map $\tau\colon\K^n\to\R$, is  asymptotically continuous.
\end{prop}

Note that the term \emph{asymptotically continuous}, which we will be using again, is short for ``continuous with respect to the asymptotic topology".
To prove the above result we  need to reveal the relation between $\rc(K)$ and $\nc(K)$. These cones are linked via the \emph{barrier cone}, $\bc(K)$, which is a convex cone defined as  the set of all vectors $v\in\R^n$ such that  $\sup_{K}\l v, \cdot\r<\infty$. It turns out that $\rc(K)$ is the \emph{polar cone} of $\bc(K)$, $\bc(K)^*=\rc(K)$,  \cite[p. 123, Cor. 14.2.1]{rockafellar}. If $C$ is a convex cone, then $C^*$ consists of all $x^*\in\R^n$ such that $\l x^*,x\r\leq 0$ for all $x\in C$. In particular, if $o\in C$, then  $C^*=\nc(C)$. It is well-known that $C^{**}=\cl(C)$ \cite[p. 125]{rockafellar},  which  yields 
$$
\rc(K)^*=\bc(K)^{**}=\cl\big(\bc(K)\big).
$$
Further $\bc(K)$ is also related to $\nc(K)$: Wu \cite[p. 283]{wu:spherical} shows that,
$
\relint\big(\ol\bc(K)\big)\subset\ol \nc(K)\subset \cl\big(\ol\bc(K)\big).
$
But, since $\ol\bc(K)$ is a convex spherical set,  $\cl(\relint(\ol\bc (K)))=\cl(\ol\bc (K))$ \cite[Lemma 3]{wu:spherical}. Thus  
$$\cl\big(\bc(K)\big)=\cl\big(\nc(K)\big).$$
The last two displayed expressions now yield:
\begin{equation}\label{eq:clbc}
\rc(K)^*=\cl\big(\nc(K)\big).
\end{equation}
Equipped with this fact, we are  ready to prove the last proposition:

\begin{proof}[Proof of Proposition \ref{lem:olnc}]
Suppose that  there exists a sequence $K_i\in\K^n$ such that $K_i\overset{a}{\to} K\in\K^n$. Then we have to show that $\cl(\nc(K_i))\overset{a}{\to} \cl(\nc(K))$, which is equivalent to $\cl(\nc(K_i))\overset{bh}{\to} \cl(\nc(K))$, since we are dealing with cones. The latter convergence in turn can be rewritten as
$\rc(K_i)^*\overset{bh}{\to}\rc(K)^*$, by \eqref{eq:clbc}. So we just need  to check that
the polar cone mapping $\C^n\ni C\mapsto C^*\in\C^n$ is continuous with respect to the bounded-Hausdorff topology, which is a known fact, e.g., see \cite[11.35(b)]{rw}.
\end{proof}

\begin{note}
Another  topology which lies in between bounded-Hausdorff and Hausdorff is the \emph{cosmic topology} which is studied  by Rockafellar and Wets \cite{rw}; however, this topology does not entail the continuity of the recession cones, or total curvature. In particular the example of parabolas $y\geq tx^2$ mentioned earlier will converge to the upper half plane under the cosmic topology by \cite[Thm. 4.25(c)]{rw}.
\end{note}

\section{The Central  Direction}\label{sec:crd}
As we mentioned in the introduction, the first step in proving Theorems \ref{thm:main} and \ref{thm:main2} is to assign a certain  unit vector $u\in\S^{n-1}$ to each  $K\in\K^n$.  To find this vector  set $\ol{\rc(K)}:=\rc(K)\cap\S^{n-1}$.
   We call any  vector $u\in\ol{\rc(K)}$  a \emph{recession direction}, or simply a \emph{direction} of $K$. Further, we say that $K$ has \emph{balanced support} with respect to  some vector $u\in\S^{n-1}$ if: (i) $-u\in\nc(K)$, i.e., $K$ has a support hyperplane $\d H$ with outward normal $-u$; and (ii) $\d H\cap K$  contains half a line only when it contains the whole line.   If  in addition (iii) $u\in\ol\rc(K)$, 
   then we say that $u$ is a \emph{balanced direction} of $K$.
    It follows from a result of Wu \cite[Thm. 2]{wu:spherical} that each convex body $K\in\K^n_+$ has a balanced direction $u$. Here we prove the existence of $u$ for all $K\in\K^n$, and show that $u$ may be chosen canonically.
Let the \emph{central direction} of $K$  be  the normalized center of mass or average of the recession directions, i.e., set
$$
\cd(K):=\frac{\int_{\ol\rc(K)}x\,d\omega_{m-1}}{\|\int_{\ol\rc(K)}x\,d\omega_{m-1}\|},
$$
where $d\omega_{m-1}$ denotes the volume element of $\S^{m-1}$, and $m$ is the dimension of the affine hull of $\rc(K)$. The main result of this section is:

\begin{prop}\label{prop:direction}
Every $K\in\K^n$ has a well-defined central direction $\cd(K)$. Furthermore $cd(K)$ is balanced, and $\K^n\overset{\cd}{\longmapsto}\S^{n-1}$ is asymptotically continuous.
\end{prop}

Note that $\cd(K)$, if it is well defined, depends continuously on $\rc(K)$ which  depends continuously on $K$ with respect to the asymptotic topology. Thus  $\cd\colon\K^n\to\S^{n-1}$ would be asymptotically continuous.
We will complete the rest of the proof of the above proposition in two parts: first we check  that $\cd(K)$ is well-defined (Section \ref{subsec:crd1}), and then show that  $\cd(K)$ is balanced (Section \ref{subsec:crd2}). 

\subsection{Existence}\label{subsec:crd1}
Here we show that every $K\in\K^n$ has a  well-defined central direction, i.e., we check that $\int_{\ol\rc(K)}x\,d\omega_{m-1}\neq o$. To see this first observe that $\cd(K)=\cd(\rc(K))$, i.e., $K$ has a central direction if and only if its recession cone has a central direction. Then it is enough to show that the recession cone of every $K\in\K^n$ is relatively proper (Lemma \ref{lem:relbd}), and every relatively proper convex cone has a central direction (Lemma \ref{lem:cdcone}). These arguments require a simple observation:
 
  \begin{lem}\label{lem:projection}
 Let $K\subset\R^n$ be a convex body, $L\subset\R^n$ be a linear subspace, and $\pi\colon\R^n\to L$ be the orthogonal projection. Then
 $\rc(\pi(K))=\pi(\rc(K))$. 
   \end{lem}
   \begin{proof}
   We may suppose that $o\in K$. Then $\rc(K)\subset K$, and so $\pi(\rc(K))\subset \pi(K)$, which  implies that $\rc(\pi(\rc(K)))\subset \rc(\pi(K))$. But $\rc(\pi(\rc(K)))=\pi(\rc(K))$, since $\pi(\rc(K))$ is a closed cone. So we conclude that $\pi(\rc(K))\subset \rc(\pi(K))$. Next, we establish the reverse inclusion. This is trivial when $\rc(\pi(K))=o$; otherwise,  let $u\in\ol\rc(\pi(K))$, and consider the half-line $t u$ in $\pi(K)$, $t\geq 0$. Now for every $t>0$ let $v_t\in K$ be a point with $\pi(v_t)=t u$, and set $\ol v_t:=v_t/\|v_t\|$. Since $\S^{n-1}$ is compact, there exists a subsequence $\ol v_i:=\ol v_{t_i}$ which converges to $\ol v\in\S^{n-1}$. Then the half-line $t\ol v$ lies in $K$; because it is a limit of line segments $ov_i$ which lie in $K$ (since $K$ is convex). But $\pi(\ol v)=u$, since $\pi(\ol v_i)=u$. Thus $\pi(t\ol v)=tu$. We have shown then that any half-line $tu$ in $\pi(K)$ is the image under $\pi$ of a half-line $t\ol v$ in $K$. So $\rc(\pi(K))\subset\pi(\rc(K))$, which completes the proof.
   \end{proof}

Using the last lemma, we can now establish the following characterization. Recall that a subset of $\R^n$ is relatively proper if it is a proper subset of its affine hull.
 
\begin{lem}\label{lem:relbd}
Let $K\subset\R^n$ be an unbounded convex body. Then $K\in\K^n$, if and only if  $\rc(K)$ is  relatively proper.
\end{lem}
\begin{proof}
 Recall that, by Lemma \ref{lem:rc4}, $K=\ol K+L$, where $L$ is the maximal linear subspace of $\rc(K)$, and $\ol K$ is the image of $K$ under the orthogonal projection $\pi\colon\R^n\to L^\perp$. By Lemma \ref{lem:projection}, $\rc(\ol K)=\rc(\pi(K))=\pi(\rc(K))$. So, by Lemma \ref{lem:rc1}, $\ol K$ is unbounded if and only if $\pi(\rc(K))\neq o$, or  $\rc(K)\not\subset L$, which means that $\rc(K)$ is relatively proper. Further, again by Lemma \ref{lem:rc4}, $\ol K$ is unbounded if and only if $K\in\K^n$. So we conclude that $\rc(K)$ is relatively proper if and only if $K\in\K^n$.
\end{proof}

Now recall  that $\cd(K)=\cd(\rc(K))$, and, by Lemma \ref{lem:relbd}, $\rc(K)$ is relatively proper when $K\in\K^n$. Thus to show that $\cd(K)$ is well-defined, it suffices to establish:

\begin{lem}\label{lem:cdcone}
Every relatively proper convex cone $C\subset\R^n$ has a central direction.
\end{lem}
\begin{proof}
Note that $\rc(C)=C$ and set $\ol C:=\ol\rc(C)=C\cap \S^{n-1}$.
Suppose that $\dim(C)=m$. Then, after a rotation, we may assume that $C$ lies in the space of the first $m$ coordinates $\R^m\times \{o_{n-m}\}\subset\R^n$, which we identify with $\R^m$. Consequently $\ol C$ lies in $\S^{m-1}$ and has interior points there. By Lemma \ref{lem:relbd}, $o\in\relbd(C)$, the relative boundary of $C$, for otherwise, we would have $o\in\relint(C)$, the relative interior of $C$, which, since $C$ is a cone would  imply that $C$ fills its affine hull and so is not relatively proper. Thus, since $C$ is convex, it follows that there exists a supporting hyperplane $\d H\subset\R^m$ of $C$ which passes through $o$. Since $\ol C$ has interior points in $\S^{m-1}$, it cannot lie entirely in $\d H$, and consequently we can choose a unit vector $u\in\S^{m-1}$ which is orthogonal to $\d H$ and points towards the side of $\d H$ in $\R^m$ where $\ol C$ lies.
Then the height function $\l u, \cdot\r\geq 0$ on $\ol C$ and $\l u, \cdot\r> 0$ on a subset of $\ol C$ which is open in $\S^{m-1}$. Thus 
$$
0<\int_{\ol C}\l u,x\r\,d\omega_{m-1}=\left\l u,\int_{\ol C}x\,d\omega_{m-1}\right\r,
$$
which shows that $\int_{\ol C}x\,d\omega_{m-1}\neq o$, and consequently $\cd(C)$ is well-defined.
\end{proof}

\subsection{Balance}\label{subsec:crd2}
To complete the proof of Proposition \ref{prop:direction}, it remains only to check that $\cd(K)$ is balanced. Once again we reduce this claim to a corresponding statement about recession cones, Lemma \ref{lem:proper}, which states that the central direction of a relatively proper convex cone is always balanced.  To establish this fact, we need a basic property of  spherical sets described in the next lemma.
Here a  set $X\subset\S^n$ is \emph{convex} provided that every pair of points of $X$ may be joined by a distance minimizing geodesic segment (or piece of a great circle) which is contained in $X$. Note that $X$ is convex if and only if the cone over $X$, i.e., the set of all half-lines $\lambda x$, where $x\in X$ and $\lambda\geq 0$, is convex and is not a line. The following basic fact refines an earlier observation about convex spherical sets  \cite[Prop. 2.1]{cgr}.

\begin{lem}\label{lem:sphere}
Let $X\subset\S^n$ be  convex,  $x_0\in X$, and  $X_+$, $X_-$ be the subsets of $X$ where $\l x_0, \cdot\r\geq 0$ and $\leq 0$ respectively. Further let $X_-'$ be the reflection of $X_-$ with respect to the hyperplane orthogonal to  $x_0$ which passes through the origin. Then $X_-'\subset X_+$.
\end{lem}
\begin{proof}
The case $n=1$ is  trivial. So let us assume that $n\geq 2$. Further
we may assume, for convenience, that $x_0$ is the ``north pole" $e_n:=(0,\dots,0,1)$. Let $\S^n_+$ and $\S^n_-$ denote, respectively, the ``northern" and ``southern" hemispheres of $\S^n$, i.e., collection of  $x\in \S^n$ where $\l x, e_n\r\geq 0$ or $\leq 0$ respectively. Also let $E:=\S^n_+\cap \S^n_-$ denote the ``equator". Now let
$y\in X_-$. We have to show that $y'\in X_+$, where $y'$ is the reflection of $y$ with respect to the hyperplane of the first $n$-coordinates. If $y\in E$, then $y'=y\in E\cap X_-=E\cap X_+\subset X_+$ and we are done. Further, if $y=-e_n$ then $y'=e_n\in X_+$ and again we are done. So suppose that $y$ lies in the interior of $\S^n_-$, and is different from $-e_n$. Then there exists a unique geodesic segment $\Gamma$ joining $y$ and $e_n$, see Figure \ref{fig:sphere}.
\begin{figure}[h]
   \centering
   \begin{overpic}[width=1.25in]{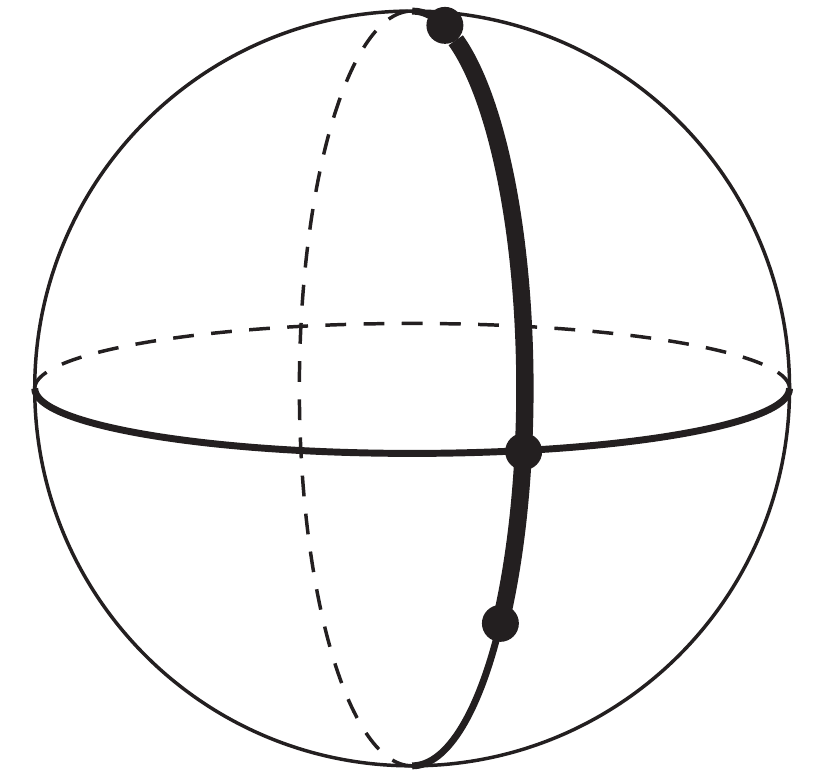} 
   \end{overpic}
   \put(-48,40){$m$}
   \put(-32,12){$y$}
   \put(-52,90){$e_n$}
   \put(-30,62){$\Gamma$}
   \put(-100,41){$E$}
   \caption{}
   \label{fig:sphere}
\end{figure}
Also note that, since $y$ and $e_n$ lie in the interior of opposite hemispheres, there exists a point $m$ of $\Gamma$ strictly between $y$ and $e_n$ which belongs to $E$. Since $\Gamma$ is a piece of a great circle $C$, it lies in the intersection of $\S^n$ with the two dimensional plane $\Pi$ containing  $e_n$, $m$, and the origin $o$ ($\Gamma\subset C:=\Pi\cap \S^n$). Since $o'=o$, $m'=m$ and $e_n'=-e_n\in \Pi$, it follows that $\Pi'=\Pi$. So $C'=C$, which shows that $y'\in C$.  Further, since $y$ lies in the shorter of  the two segments of $C$ between $-e_n$ and $m$,  $y'$ lies  in the shorter of the two segments of $C$ between $m'=m$ and $-e_n'=e_n$. So $y'\in \Gamma$, which, since $X$ is convex, implies that $y'\in X$. Of course, since $y\in\S^n_-$, we also have $y'\in\S^n_+$. So $y'\in \S^n_+\cap X=X_+$.
\end{proof}

Using the above lemma we can now show:

\begin{lem}\label{lem:proper}
If $C\subset\R^n$ is a relatively proper convex cone, then $\cd(C)$ is balanced.
\end{lem}
\begin{proof}
Recall that $\ol C:=C\cap \S^{n-1}$, and note that for any fixed $x_0\in \ol C$,
$$
\left\l x_0, \int_{\ol C}x\,d\omega_{m-1}\right\r=\int_{\ol C} \l x_0, x\r \,d\omega_{m-1}=\int_{\ol C_+} \l x_0, x\r \,d\omega_{m-1}+\int_{\ol C_-} \l x_0, x\r \,d\omega_{m-1},
$$
where $\ol C_+$ is the portion of $\ol C$ which is contained in the hemisphere centered at $x_0$ and $\ol C_-$ is the portion contained in the opposite hemisphere. So $\l x_0, x\r\geq 0$ on $\ol C_+$ and $\l x_0, x\r\leq 0$ on $\ol C_-$. Next note that
$$
\int_{\ol C_-} \l x_0, x\r \,d\omega_{m-1}=-\int_{\ol C_-'} \l x_0, x\r \,d\omega_{m-1},
$$
where $\ol C_-'$ denotes the reflection of $\ol C_-$ with respect the hyperplane orthogonal to $x_0$ which passes through the origin. But, since $C$ is a convex cone which is not a line, $\ol C$ is a convex subset of $\S^{n-1}$; therefore, $\ol C_-'\subset \ol C_+$ by Lemma \ref{lem:sphere}. So, 
$$
\left\l x_0, \int_{\ol C}x\,d\omega_{m-1}\right\r=\int_{\ol C_+-\ol C_-'} \l x_0, x\r \,d\omega_{m-1}\geq 0.
$$
This shows that $x_0$ and consequently $\ol C$ lie in the hemisphere centered at $\cd(C)$. So we conclude that the hyperplane $\d H$ which passes through the origin and is orthogonal to $\cd(C)$ supports $C$.

It remains only to check that $\d H\cap C=-\d H\cap C$. If $\d H\cap C=o$, then we are done; otherwise, 
there exists a unit vector $x_0\in \d H$ such that $\l x_0, \int_{\ol C}x\,d\omega_{m-1}\r=0$. Consequently the last displayed expression above implies that  $\ol C_+=\ol C_-'$. This in turn yields that $-x_0\in \d H\cap C$. Thus, since $\d H\cap C$ is a cone, it follows that $\d H\cap C=-\d H\cap C$.
\end{proof}

Recall that, by Lemma \ref{lem:relbd}, if $K\in\K^n$, then $\rc(K)$ is relatively proper. Furthermore $\cd(\rc(K))=\cd(K)$. Thus 
the last lemma shows that  $\cd(K)$ is a balanced direction of $\rc(K)$.
So  there exists a support hyperplane $\d H$ of $\rc(K)$ which is orthogonal to $\cd(K)$, and is balanced, i.e., $\d H\cap\rc(K)=-\d H\cap\rc(K)$. Now let $\d H'$ be the support hyperplane of $K$  which is orthogonal to $\cd(K)$. Then $\d H$ and $\d H'$ are parallel.
So, by Lemma \ref{lem:rc1}, $\d H'\cap K$ contains a half-line $\ell$ if and only if $\d H\cap\rc(K)$ contains a half-line parallel to $\ell$, which yields that $\cd(K)$ is a balanced direction of $K$. 

\begin{note}
The mapping $\K^n\ni K\overset{\cd}{\longmapsto} \cd(K)\in\S^{n-1}$ is not continuous with respect to the bounded-Hausdorff topology. To see this, let $K_t\subset\R^2$ be the family of convex bodies given by $x\leq t$ and $y\geq 0$. Then $\cd(K_t)=(-1,1)/\sqrt{2}$ for all $t$, see Figure \ref{fig:wedge}; however, $K_t\overset{bh}{\to} \mathbf{H}^2$, as $t\to\infty$, where $\mathbf{H}^2$ is the upper half plane, and $\cd(\mathbf{H}^2)=(0,1)$.
\begin{figure}[h]
\begin{center}
\begin{overpic}[width=4.5in ]{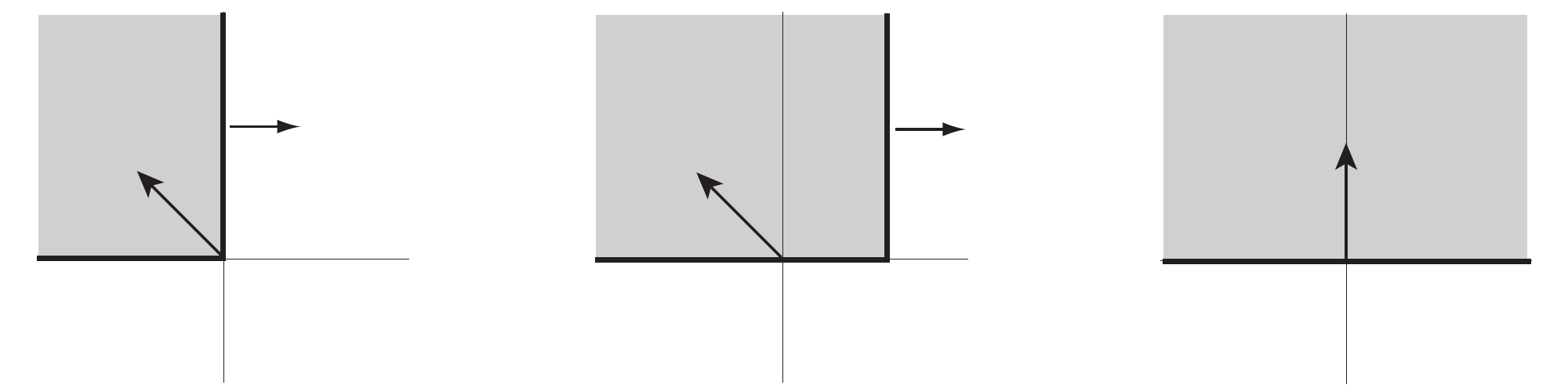}
\end{overpic}
\put(-265,10){$K_0$}
\put(-150,10){$K_1$}
\put(-32,10){$K_\infty$}
\caption{}
\label{fig:wedge}
\end{center}
\end{figure}
\end{note}

\section{Topology and Continuity of Minkowski Addition}\label{sec:continuous}

As we mentioned in the introduction, in general the Minkowski sum of a pair of convex bodies $K_0$, $K_1\in\K^n$ does not belong to $\K^n$. For instance the sum of a pair of closed half-spaces of $\R^n$ whose boundaries are not parallel is the whole $\R^n$. Furthermore, the Minkowski addition in not a continuous operation on the space of convex bodies, even with respect to the bounded Hausdorff topology (see Note \ref{note:cont}).
Here we will derive some conditions for Minkowski addition to operate properly on $\K^n$ and be continuous with respect to the asymptotic topology. For any direction $u\in\S^{n-1}$, let  $\K^n_u$ denote the collection of convex bodies $K\in\K^n$ 
which have balanced support with respect to $u$ (as we defined in Section \ref{sec:crd}).
Further, set $(\K^n_+)_u:= \K^n_+\cap \K^n_u$. We show that Minkowski addition acts continuously on this space:

\begin{prop}\label{prop:continuous}
For any pair of convex bodies $K_0$, $K_1\in (\K^n_+)_u$, $K_0+K_1\in (\K^n_+)_u$. Furthermore, the Minkowski addition $+\colon (\K^n_+)_u\times(\K^n_+)_u\to (\K^n_+)_u$ is asymptotically continuous.
\end{prop}

The proof of this result  requires a few lemmas. The first one below establishes a compact version of Proposition \ref{prop:continuous}. Note that the sum of a pair of compact sets is always compact, so the operation below is well-defined.

\begin{lem}\label{lem:X}
In the space $X$ of compact subsets of $\R^n$ the minkowski sum $+\colon X\times X\to X$ is continuous with respect to the Hausdorff topology.
\end{lem}
\begin{proof}
Let $A$, $B\in X$, and $A_i$, $B_i\in X$ be sequences such that $A_i\overset{h}{\to} A$ and $B_i\overset{h}{\to} B$. We have to show that then $A_i+B_i\overset{h}{\to} A+B$. To see this, for any subset $A$ of $\R^n$ let $A_r:=A+r\mathbf{B}^n$. Next note that  for every $\epsilon >0$, there exists an integer $N$ such that $A_i\subset A_{\epsilon/2}$ and $B_i\subset B_{\epsilon/2}$ for $i\geq N$. Thus, since $r\mathbf{B}^n +r \mathbf{B}^n=2r\mathbf{B}^n$, it follows that $A_i+B_i\subset (A+B)_{\epsilon}$ which completes the proof.
\end{proof}

We now use the last lemma  to obtain an asymptotic version of Proposition \ref{prop:continuous}. Again it is easy to check that the sum of closed cones is always a closed cone, so the addition operation here is well-defined.

\begin{lem}\label{lem:conecont}
Let $\C$ be the space of  nontrivial closed convex cones in $\R^n$ which lie in the upper half space $\mathbf{H}^n$, and $\C^+\subset \C$ consist of those cones which intersect $\d \mathbf{H}^n$ only at the origin; then  the Minkowski sum $+\colon \C\times \C^+\to \C$ is continuous under bounded-Hausdorff topology. 
\end{lem}
\begin{proof}
For every  $C\in \C$, let $\ol C:=C\cap\S^{n-1}$ and note that for any family $C_i\in\C$, $C_i\overset{bh}\to C$ if and only if $\ol C_i\overset{h}\to \ol C$. Now recall that the spaces $\ol \C$ and $\ol\C^+$ consist of convex spherical sets and are in one-to-one correspondence with the spaces $\C$ and $\C^+$ via the operation $C\mapsto \ol C$, whose inverse is obtained by taking the cones over the elements $\ol C\in \ol\C$. So all we need to show then is that 
$\overset{\circ}{+}\colon \ol\C\times \ol\C^+\to \ol\C$ is continuous under Hausdorff topology, where $C\overset{\circ}{+}C'$ is the collection of all $x\overset{\circ}{+}y:=(x+y)/\|x+y\|$ such that $x\in C$ and $y\in C'$. Since elements of $\ol\C$ all lie in the same hemisphere, and the elements of $\ol\C^+$ do not touch the boundary of that hemisphere,  there exists no pair of points $x\in C\in\ol\C$ and 
$y\in C'\in\ol\C^+$ such that $x=-y$; thus $\overset{\circ}{+}\colon \ol\C\times \ol\C^+\to \ol\C$ is well defined. Finally, let $\pi\colon \R^n-\{o\}\to \S^{n-1}$ be  given by $\pi(x):=x/\|x\|$, and note that $x\overset{\circ}{+}y=\pi(x+y)$. Thus, since $\pi$ is continuous, Lemma \ref{lem:X} implies that $\overset{\circ}{+}\colon \ol\C\times \ol\C^+\to \ol\C$ is continuous, as desired.
\end{proof}

Next we establish the additive property of the recession cones:
 
 \begin{lem}\label{lem:rcsum}
For any pairs of convex bodies $K_0$, $K_1\subset\R^n$, $$\rc(K_0+K_1)=\rc(K_0)+\rc(K_1).$$
\end{lem}
\begin{proof}
Set $K:=K_0+K_1$.
We may assume that $K_0$, $K_1$ both contain the origin. Then $K_0$,
$K_1\subset K$, which implies  that $\rc(K_0)$, $\rc(K_1)\subset\rc(K)$. So, since cones are closed under Minkowski addition, 
 $\rc(K_0)+\rc(K_1)\subset\rc(K)$, which completes half of the proof. To prove the reverse inclusion, we may suppose that there exists a vector $v\in \rc(K)-\{o\}$, for else there is nothing to prove. Then
$$
\infty=\sup_{K}\l v,\cdot\r=
\sup_{K_0}\l v,
\cdot\r+\sup_{K_1}\l v,
\cdot\r,
$$
which implies that $v\in\rc(K_0)\cup\rc(K_1)\subset\rc(K_0)+\rc(K_1)$.
\end{proof}

We also need the following simple characterization:

\begin{lem}\label{lem:kpluscompact}
Let $K\subset\R^n$ be an unbounded convex body. Then $K\in\K^n_+$ if and only if $K$ contains no lines. \end{lem}
\begin{proof}
If $K\in\K^n_+$, then  there is a support hyperplane $\d H$ of $K$ such that $K\cap\d H=\{p\}$. So $K$ has no lines passing through $p$, which implies that   $K$ contains no lines,  by Lemma \ref{lem:rc1}. Conversely, if $K$ contains no lines, then $K\in\K^n$ by Lemma \ref{lem:rc4}; because if $\dim(L)=0$, then $K=\ol K$ in Lemma \ref{lem:rc4}, which yields that $\ol K$ is unbounded. 
It only remains then to check that $K$ has a strictly convex point. To see this note that
 $\rc(K)$ contains no half-lines parallel to $\d H$, since $K\cap\d H$ is compact. So it follows that $\d H'\cap K$ is also compact  for any hyperplane $\d H'$ parallel to $\d H$. Take one such hyperplane $\d H'$ which is different from $\d H$ and intersects $K$, see Figure \ref{fig:tub}.
 \begin{figure}[h]
\begin{center}
\begin{overpic}[width=3in ]{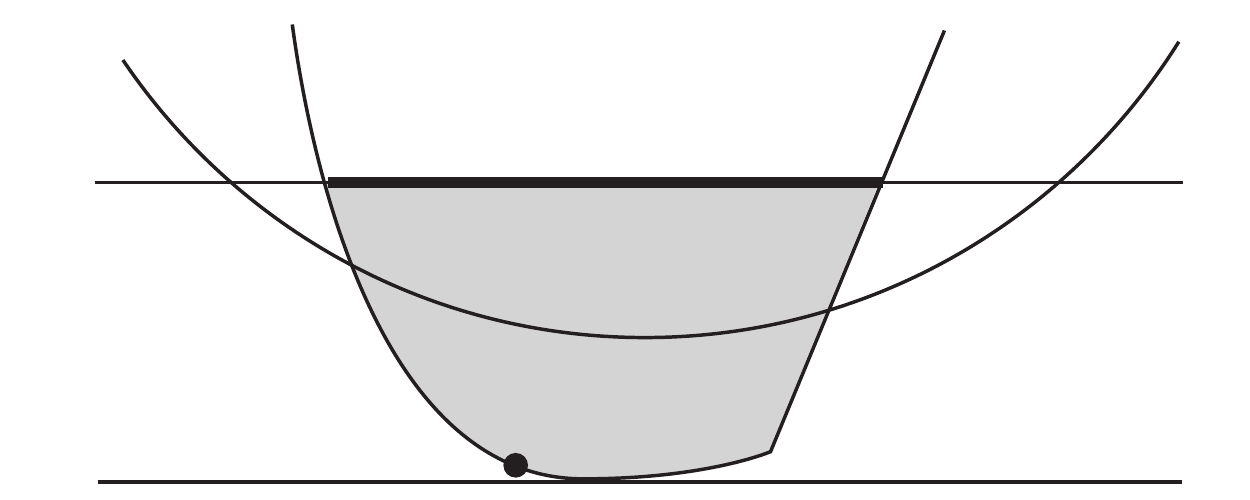}
\end{overpic}
\put(-220,2){$\d H$}
\put(-220,53){$\d H'$}
\put(-118,37){$\ol K$}
\put(-40,65){$B$}
\put(-130,12){$p$}
\caption{}
\label{fig:tub}
\end{center}
\end{figure}
 Further let $H'$ be the half-space determined by $H'$ which contains $\d H$, and set $ K':=K\cap H'$. Then $ K'$ is compact, since if it contains any half-lines, then they must be parallel to $\partial H$, which is impossible since $\rc(K)$ contains no half-lines parallel to $\partial H$ as we have already discussed. Further since $\d H'\cap K$ is compact and is disjoint from $\d H$,  there  exists a ball $B\subset\R^n$ which contains $\d H'\cap K$ but is disjoint from $\d H$. Let $p$ be  the farthest point of $ K'$ from the center of $B$ (which exists by compactness of $K'$). Then $p$ is a strictly convex point of $ K'$. We claim that  $p$ is a strictly convex point of $K$ as well. 
  To see this  note that $p\not\in B$, since $\d H\cap  K'\not\subset B$.  In particular $p\not\in\d H'$, since $\d H'\cap K'\subset B$. Thus $p$ lies in the interior of $H'$, $p\in\inte(H')$, which shows that
$$
p\in  K'\cap \inte(H')= K\cap \inte(H').
$$
On the other hand, since $p$ is a strictly convex point of $ K'$, there exists a hyperplane $\Pi$ such that $\Pi\cap  K'=\{p\}$. So we have
$$
p\in\Pi\cap K\cap \inte(H')\subset \Pi\cap  K'=\{p\}.
$$
Thus $\Pi\cap K\cap \inte(H')=\{p\}$. Now suppose, towards a contradiction, that $\Pi$ intersects $K$ at some other point $q$. Then the line segment $pq$ lies in $\Pi\cap K$, by convexity of $\Pi$ and $K$. But $K\cap \inte(H')$, which is an open neighborhood of $p$ in $\R^n$, must contain some point of $pq$ other than $p$, and therefore it must contain some point of $\Pi$ other than $p$, which is a contradiction. So $K$ has a strictly convex point.
\end{proof}

Finally we observe an important compactness property of the elements of $(K^n_+)_u$:

\begin{lem}\label{lem:technical}
Let $K\in (K^n_+)_u$ and $H$ be a half-space such that $u$ is the outward normal to $\partial H$. Then $K\cap H$ is compact.
\end{lem}
\begin{proof}
We may assume that $o\in\partial H$. Then
$$\rc(K\cap H)=\rc(K)\cap \rc(H)=\rc(K)\cap H.$$ Next note that
$\l \cdot ,u\r\leq 0$ on $H$, since $\partial H$ has outward normal $u$.    On the other hand,  $K$ is supported by a hyperplane 
$\partial H'$ with outward normal $-u$ by assumption. So $\l \cdot,u\r\geq 0$  on $\rc(K)$.  Further note that $K\cap \d H'$ is compact: for otherwise it must contain a half-line by Lemma \ref{lem:rc1}, and therefore a full line by the balance assumption, which is not permitted by Lemma \ref{lem:kpluscompact}. So  $\rc(K)$ does not contain any half-lines orthogonal to $u$.
 Consequently $\l \cdot,u\r >0$ on $\rc(K)-\{o\}$. So $\rc(K)\cap H=\{o\}$, which finishes the proof by Lemma \ref{lem:rc1}.
\end{proof}

Now we are ready to prove the last proposition:

\begin{proof}[Proof of Proposition \ref{prop:continuous}]
First we check that if $K_0$, $K_1\in (\K^n_+)_u$, then $K:=K_0+K_1\in (\K^n_+)_u$.  The sum of convex sets is always convex, so the convexity of $K$ is automatic. Next we check that $K$ is closed (which is not automatic by Note \ref{note:closed}). For convenience let $u=(0,\dots,0,1)$. Then after translations we may assume that $K_0$, $K_1$ lie in the the upper half-space $x_n\geq 0$ and are supported by the hyperplane $x_n=0$. Now let $H^t$ be the half-space given by $x_n\leq t$, and note that
$$
K\cap H_t=(K_0\cap H_t+K_1\cap H_t)\cap H_t.
$$ 
Further recall that, by Lemma \ref{lem:technical}, $K_0\cap H_t$ and $K_1\cap H_t$ are compact. But  sum of compact sets is always compact. So $K\cap H_t$ is compact for all $t$, which yields that $K$ is closed. Now we know that $K$ is a convex body, since it clearly has interior points. Further note that since $K$ is supported by $\d H_0$ and $K\cap \d H_0$ is bounded, $K$ contains no lines (by Lemma \ref{lem:rc1}). 
 So $K\in \K^n_+$ by Lemma \ref{lem:kpluscompact}. Finally, again since $K\cap\d H_0$ is compact, it follows that $K$ has balanced support with respect to $u$, so  $K\in (\K^n_+)_u$. 

Next we establish the continuity of the Minkowski addition on $(\K^n_+)_u$. To see this, let  $K^i_0$, $K^i_1$ be sequences in  $(\K^n_+)_u$ such that $K^i_0 \overset{a}{\to} K_0$ and $K^i_1 \overset{a}{\to} K_1$, where recall that $\overset{a}{\to}$ denotes convergence with respect to the asymptotic topology.
Setting $K^i:=K^i_0+K^i_1$, we need to show   that
$$
K^i\overset{a}{\longrightarrow} K.
$$
To establish this convergence we need in turn to verify  
$$
(a)\;\rc(K^i)\overset{bh}{\longrightarrow}\rc(K)\quad\quad\text{and}\quad\quad
(b)\;K^i\overset{bh}{\longrightarrow}K,
$$
where, recall that, $\overset{bh}{\to}$ indicates convergence with respect to the bounded-Hausdorff topology.  To see (a)  note that by Lemma \ref{lem:rcsum},
$$
\rc(K^i)=\rc(K^i_0)+\rc(K^i_1),\quad\text{and}\quad\rc(K)=\rc(K_0)+\rc(K_1).
$$ 
Furthermore, by assumption
$$
\rc(K^i_0)\overset{bh}{\longrightarrow}\rc(K_0),\quad\text{and}\quad \rc(K^i_1)\overset{bh}{\longrightarrow}\rc(K_1).
$$
Since we have assumed $u=(0,\dots,0,1)$, all  recession cones are supported by the hyperplane $\d \mathbf{H}^n$, and Lemma \ref{lem:conecont} finishes the proof of (a).
Next, to verify (b), it is enough to show that 
\begin{equation}\label{eq:Ki}
K^i\cap H_t\overset{h}{\longrightarrow} K\cap H_t.
\end{equation} 
where $
H_t$ are the half-spaces given by $x_n\leq t$ defined earlier. Also note that, similar to the earlier argument, we have 
$$
K^i\cap H_t=(K^i_0 \cap H_t+ K^i_1 \cap H_t)\cap H_t.
$$
Further $K^i_j \cap H_t\overset{h}{\to} K_j \cap H_t$, $j=1$, $2$, since $K^i_j \cap H_t$ are compact by Lemma \ref{lem:technical} and $K^i_j \overset{bh}{\longrightarrow}K_j$. So, by Lemma \ref{lem:X}, 
$$
K^i_0 \cap H_t+ K^i_1 \cap H_t\overset{h}{\longrightarrow} K_0 \cap H_t+ K_1 \cap H_t.
$$ 
The last two displayed expressions now imply \eqref{eq:Ki}.
\end{proof}

Finally we need  a variation of the last proposition, which establishes the continuity of the addition when one of the summands is a fixed element of $K^n_u$. This result is sharp since Minkowski addition is not continuous on $K^n_u$, see Note \ref{note:knucont}.

\begin{prop}\label{prop:continuous2}
For any pair of convex bodies $K_0\in \K^n_u$, and $K_1\in (\K^n_+)_u$, $K_0+K_1\in \K^n_u$. Furthermore, 
for any fixed $K_0\in \K^n_u$, the mapping $K_0+(\cdot)\colon (\K^n_+)_u\to \K^n_u$ is asymptotically continuous.
\end{prop}
\begin{proof}
First check that if $K_0\in \K^n_u$, and $K_1\in (\K^n_+)_u$, then $K:=K_0+K_1\in \K^n_u$. Recall that, by Lemma \ref{lem:rc4}, $K_0=\ol K_0+L_0$ where $L_0$ is the linearity space of $K_0$ and $\ol K_0$ is the  projection of $K_0$ into $L_0^\perp$. Further note that, since $\ol K_0$ has no lines, Lemma \ref{lem:kpluscompact} implies that $\ol K_0\in\K^{n-m}_+$ where $m$ is the dimension of $L_0$, and so we may identify $L_0^\perp$ with $\R^{n-m}$. Now let $\ol K_1$ be the projection of $K_1$ into $L_0^\perp$. Then $\ol K_1\in \K^{n-m}_+$ as well, because $K_1$ has no lines (again by Lemma \ref{lem:kpluscompact}), and so $\ol K_1$ has no lines. Further, since by assumption $K_0$, $K_1$ have balanced support with respect to $u$, then so do the projections $\ol K_0$ and $\ol K_1$. So 
 $\ol K_0$, $\ol K_1\in (\K^{n-m}_+)_u$, which implies, by Proposition \ref{prop:continuous}, that $\ol K=\ol K_0+\ol K_1\in (\K^{n-m}_+)_u$.
 Consequently $K=L_0+\ol K\in \K^n_u$.

It remains to verify the continuity of the addition which, as in the proof of Proposition \ref{prop:continuous}, consists of checking the bounded-Hausdorff convergence of the recession cones, followed by the bounded-Hausdorff convergence of the bodies. Convergence of the cones again follows from Lemma \ref{lem:conecont}.
To see the convergence of the bodies,  let $K^i_1\in (K^n_+)_u$ be a family of convex bodies such that $K^i_1\overset{bh}{\to}K_1$. We have to show then that $K_0+K^i_1\overset{bh}{\to}K$. Now let $\ol K_0$, $L_0$, be as in the last paragraph, and $\ol K_1^i$ be the projection of $K_1^i$ into $L_0^\perp$. Then
\begin{equation}\label{eq:K0Ki1}
K_0+K^i_1=L_0+ \ol K_0+\ol K_1^i.
\end{equation}
But, as we argued in the last paragraph, $\ol K_0$, $\ol K_1^i\in (\K^{n-m}_+)_u$. Further $\ol K^i_1\overset{bh}{\to}\ol K_1$, since the projection $\R^n\to L^\perp$ is continuous in the bounded-Hausdorff sense. So, by Proposition \ref{prop:continuous},
$$
\ol K_0+\ol K_1^i\overset{bh}{\longrightarrow} \ol K_0+\ol K_1.
$$
Thus, using \eqref{eq:K0Ki1}, we have
$$
K_0+K^i_1 \overset{bh}{\longrightarrow} L_0+ \ol K_0+\ol K_1=L_0+\ol K=K,
$$ 
as desired.
\end{proof}

\begin{note}\label{note:cont}
The Minkowski addition is not continuous on $\K^n$, even with respect to the bounded-Hausdorff topology. Let $\ell$ denote the nonpositive  portion of the $x$-axis in $\R^2$, and $\ell'_t$, $0\leq t\leq 1$, be the family of half-lines given by $y=tx$, $x\geq 0$; see Figure \ref{fig:wedge2}.
\begin{figure}[h]
\begin{center}
\begin{overpic}[width=3.5in ]{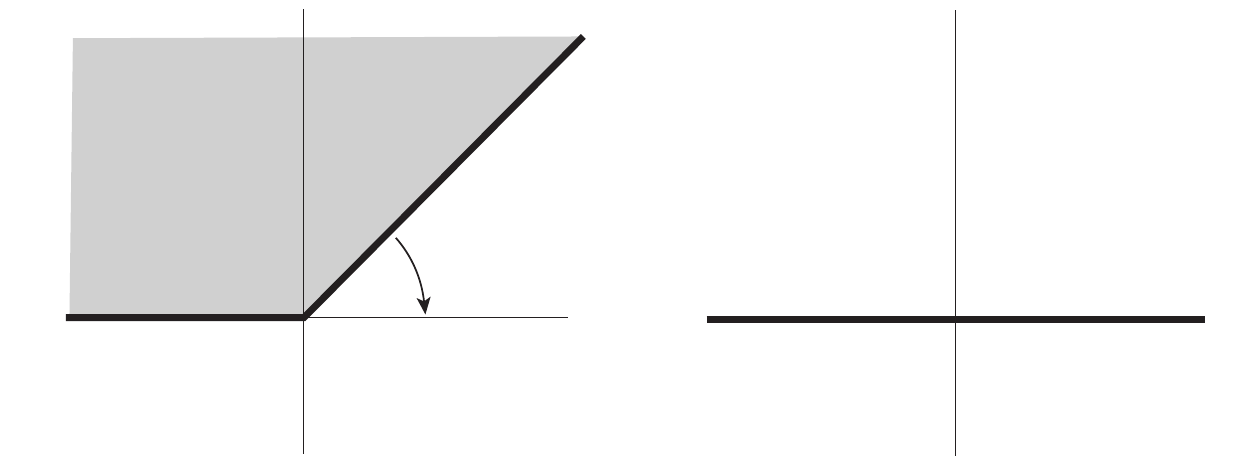}
\end{overpic}
\put(-220,16){$\ell$}
\put(-155,55){$\ell'_1$}
\put(-90,16){$\ell$}
\put(-40,16){$\ell'_0$}
\caption{}
\label{fig:wedge2}
\end{center}
\end{figure}
Now set $K:=\ell+\mathbf{B}^2$ and $K'_t:=\ell'_t+\mathbf{B}^2$. Note that $K+K'_t=\ell+\ell'_t +2\mathbf{B}^2$, and $\ell+\ell'_t$ is  the convex region bounded by the half-lines $\ell$ and $\ell'_t$ when $t>0$ Further, $\ell+\ell'_t\overset{bh}{\to}\mathbf{H}^2$, as $t\to 0$, while $\ell+\ell'_0=\R\times\{0\}$. Thus $K+K'_t\overset{bh}{\to}\mathbf{H}^2+2\mathbf{B}^2$, while $K+K'_1=\R\times\{0\}+2\mathbf{B}^2$.
\end{note}

\begin{note}\label{note:closed}
The sum of a pair of unbounded convex bodies is not in general closed. For instance let $K_0\subset\R^2$ be the convex body given by $y\geq 1/(1-x^2)$, $-1<x<1$, and let $K_1$ be the reflection of $K_0$ with respect to the $x$ axis. Then $K_0+K_1$ is the set $-2<x<2$.
\end{note}

\begin{note}\label{note:knucont}
The Minkowski addition is not continuous on $K^n_u$, even with respect to the bounded Hausdorff topology. Let $K_0\subset\R^3$ be the convex body given by $z\geq y^2$, and  $K_t$, $0\leq t\leq 1$, be the body obtained by a rotation of $K_0$ about the $z$-axis, so that $K_t$ intersects the $xy$-plane along the line $y=tx$. Then as $t\to 0$, $K_t\overset{a}{\to} K_0$; however, $K_t+K_0=\mathbf{H}^3$, for $t>0$, while $K_0+K_0=K_0$.
\end{note}

\section{ Regularity and Curvature of Minkowski Sums}\label{sec:sum}
Now we give  conditions for the Minkowski sum of convex bodies in $\K^n$  to  have the  remaining geometric and regularity properties which we need for our deformations  in Section \ref{sec:proofs}. Here \emph{curvature} refers to the Gaussian curvature. Recall that a ($C^2$) convex body $K\subset\R^n$ has positive curvature provided that the differential of its outward unit normal vector field or \emph{Gauss map} $\nu\colon\d K\to\S^{n-1}$ is nonsingular everywhere; which yields that, the \emph{principal curvatures}, i.e., the eigenvalues of the differential $d\nu$, are all positive.

\begin{prop}\label{prop:sum}
Let $K_0$, $K_1\subset \R^n$ be convex bodies, and suppose that 
$$
K:=K_0+K_1
$$ 
is closed. Then $K$ is also a convex body, and the following hold:
\begin{enumerate}
\item{If $K_0, K_1$ are strictly convex, then  so is $K$};
\item{If $K_0$  is $C^1$, then so is $K$};
\item{If  $K_0, K_1$ are $C^{2\leq k\leq\infty}$ and $K_0$ is positively curved, then $K$ is also $C^k$};
\item{If  $K_0, K_1$ are $C^\omega$ and $K_0$ is positively curved, then $K$ is also $C^\omega$};
\item{If  $K_0,K_1$  have positive  curvature, then so does $K$.}
\end{enumerate}
\end{prop}

 Some of the items in the above proposition are known or easy to establish when the convex bodies are compact, since the support functions of compact convex bodies are additive, with respect to Minkowski sums, and closely mirror the regularity of the corresponding bodies \cite[Sec. 2.5]{schneider:book}. On the other hand, in the case of the unbounded convex bodies, which is the main focus of the above proposition, we need to work harder since the support function of an unbounded convex body is not well-defined (in the conventional sense). We should also mention that the various conditions in the above proposition are  sharp; in particular see Note \ref{note:regularity}. Further, it is elementary to check that $K$ is always a convex set with interior points, and thus it is a convex body as soon as it is closed (which, unless $K_0$ and $K_1$ are compact, is not automatic as we pointed out in Note \ref{note:closed}). Finally note that the above proposition is trivially true when $K=\R^n$. So we may assume that $K$ is proper.
The enumerated items of Proposition \ref{prop:sum} will be proved in sequence in the following subsections:

\subsection{Strict convexity}\label{subsec:stconvex}
Here we check that if $K_0$ and $K_1$ are strictly convex (everywhere) then so is $K$. Recall that for any convex body $K$, $\ol\nc(K):=\nc(K)\cap\S^{n-1}$ is the unit normal cone of $K$. Now for any $u\in\ol\nc(K)$ let $\d H_u(K)$ be the support hyperplane of $K$ with outward normal $u$, then $F_u(K):=\d H_u(K)\cap K$ is the corresponding \emph{face} of $K$. 

\begin{lem}\label{lem:Fu}
For any pair of convex bodies $K_0$, $K_1\subset\R^n$, 
$$
\ol\nc(K)=\ol\nc(K_0)\cap \ol\nc(K_1),
$$
where $K:=K_0+K_1$. Furthermore, for every $u\in\ol\nc(K)$, 
$$
F_u(K)=F_u(K_0)+F_u(K_1).
$$
\end{lem} 
\begin{proof}
If $\ol\nc(K)=\emptyset$ then there is nothing to prove. Otherwise,
let $u\in \ol\nc(K)$. Then for any $x\in F_u(K)$, we have
$$
\l x, u\r=\sup_{K}\l \cdot,u\r.
$$ 
Next note that $x=x_0+x_1$, for some $x_0\in K_0$ and $x_1\in K_1$. So the last displayed expression yields that
$$
\l x_0, u\r=\sup_{K}\l (\cdot)-x_1,u\r\geq\sup_{K_0+x_1}\l (\cdot)-x_1,u\r=
\sup_{K_0}\l \cdot,u\r.
$$ 
It follows then that $u\in\ol\nc(K_0)$ and $x_0\in F_u(K_0)$. Similarly, one can show that $u\in\ol\nc(K_1)$ and $x_1\in F_u(K_1)$.
Then we have established that $\ol\nc(K)\subset\ol\nc(K_0)\cap \ol\nc(K_1)$ and $F_u(K)\subset F_u(K_0)+F_u(K_1)$. 

Conversely, suppose that $u\in \ol\nc(K_0)\cap \ol\nc(K_1)$. Then for any $x_0\in F_u(K_0)$ and $x_1\in F_u(K_1)$ we have
$$
\l x_0, u\r =\sup_{K_0} \l \cdot, u\r \quad \text{and}\quad \l x_1, u\r =\sup_{K_1} \l \cdot, u\r.
$$ 
So it follows that
$$
\l x_0+x_1, u\r=\sup_{K_0} \l \cdot, u\r+\sup_{K_1} \l \cdot, u\r= \sup_{K} \l \cdot, u\r.
$$
So $u\in\ol\nc(K)$ which completes the proof that $\ol\nc(K)=\ol\nc(K_0)\cap \ol\nc(K_1)$. Further, the last displayed expression also shows that $x_0+x_1\in F_u(K)$ and so we conclude that $F_u(K)=F_u(K_0)+F_u(K_1)$.
\end{proof}

Now note that $K$ is a strictly convex body if and only $F_u(K)$  is a singleton for all $u\in\ol\nc(K)$. Thus the above lemma quickly shows that $K$ is strictly convex whenever $K_0$ and $K_1$ are strictly convex.

\subsection{$\mathbf{C}^1$-regularity}\label{subsec:part4}
Next we check that if $K_0$ is $C^1$ then so  is $K$. To this end first we recall that a convex body  is $C^1$ if and only if through every boundary point of it there passes a \emph{unique} supporting hyperplane. This follows from the fact that locally any convex hypersurface may be represented as the graph of a convex function. More specifically,  a convex function is differentiable at a point if and only if it has only one subgradient at that point \cite[Thm 1.5.12]{schneider:book}, and the subgradient is unique if and only if the normal to epigraph of the function is unique \cite[Thm 1.5.12]{schneider:book}; further, here one also uses the fact  that a differentiable convex function is continuously differentiable \cite[Thm 1.5.2]{schneider:book}. 

Now note that if $\d K=\emptyset$, then there is nothing to prove. Otherwise, let $x\in\d K$, then $x=x_0+x_1$ for some points $x_0\in\d K_0$ and $x_1\in\d K_1$. In particular we may write $x\in K_0+x_1$. But $K_0+x_1$ is just a translation of $K_1$ and thus is $C^1$, and $x\in\d(K_0+x_1)$. Consequently, there passes only one support hyperplane of $K_0+x_1$ through $x$. On the other hand, any support hyperplane of $K$ must also support $K_0+x_1\subset K$. Thus it follows that the support hyperplane of $K$ passing through $x$ is unique. So, $K$ is $C^1$.

\subsection{$\mathbf{C}^k$-regularity}\label{sec:highreg}
Now  we show that if $K_0$ and $K_1$ are  $C^{2\leq k\leq \infty}$, and $K_0$ has  positive curvature, then $K$ is also $C^k$. First note that since $K_0$ and $K_1$ are both $C^1$, then the Gauss maps $\nu_0\colon\d K_0\to\S^{n-1}$ and $\nu_1\colon\d K_1\to\S^{n-1}$ are well-defined and are $C^{k-1}$. Further $K$ is $C^1$  by Section \ref{subsec:part4}, and so it too has a well-defined Gauss map $\nu\colon\d K\to\S^{n-1}$. Next note that if  $x\in\d K$, then by Lemma \ref{lem:Fu}, 
\begin{equation}\label{eq:x0x1}
x=x_0+x_1
\end{equation}
where $x_0\in\d K_0$, $x_1\in\d K_1$. Further  Lemma \ref{lem:Fu} implies that
 $$
 \nu_0(x_0)=\nu_1(x_1)=\nu(x).
 $$
Now,  since $K_0$ has positive curvature,  $d\nu_0$ is nonsingular and so $\nu_0\colon\d K_0\to\ol\nc(K_0)\subset\S^{n-1}$ is a $C^{k-1}$-diffeomorphism by the inverse function theorem. So \eqref{eq:x0x1}  may be rewritten as 
$$
x=\nu_0^{-1}\big(\nu_1(x_1)\big)+x_1.
$$
This suggests a possible parameterization for $\d K$. Indeed, for every $x_1\in \nu_1^{-1}(\ol \nc(K_0))$, $\nu_1(x_1)\in \ol \nc(K_0)$ and thus $\nu_0^{-1}(\nu_1(x_1))$ is well-defined. Further note  that $\nu_1^{-1}(\ol \nc(K_0))$ is open in $\d K_1$ since, by the positive curvature assumption, $\ol \nc(K_0)=\nu_0(\d K_0)$ is open in $\S^{n-1}$. So we have  a well-defined $C^{k-1}$ mapping:
\begin{equation}\label{eq:nu0}
 \nu_1^{-1}\big(\ol \nc(K_0)\big)\;\ni\; x_1\overset{f}{\longmapsto} x:=\nu_0^{-1}\big(\nu_1(x_1)\big)+x_1\;\in\;\d K.
\end{equation}
 
 We claim that $f$ is a diffeomorphism.
  First note that $f$ is onto, since if $x\in\d K$, then $\nu(x)\in\ol\nc(K)\subset\ol\nc(K_0)$ by Lemma \ref{lem:Fu}, and so $\nu_1^{-1}(\nu(x))$ lies in the domain of $f$. Thus we may compute that
  $$
  f(\nu_1^{-1}(\nu(x)))=\nu_0^{-1}(\nu(x))+\nu_1^{-1}(\nu(x))= F_{\nu_1(x)} K_0+F_{\nu_1(x)} K_1=F_{\nu_1(x)}(K),
  $$
  by Lemma \ref{lem:Fu}.
  So $f$ is onto since $F_{\nu_1(x)}(K)\ni x$. Next we check that $f$ is one-to-one. To see this  note that
 \begin{equation}\label{eq:fx1}
 f(x_1)\in F_{\nu_1(x_1)} K_0+F_{\nu_1(x_1)} K_1=F_{\nu_1(x_1)}(K),
 \end{equation}
 again by Lemma \ref{lem:Fu}. Also recall that $K$ is strictly convex by Section \ref{subsec:stconvex}. Thus the faces of $K$ are singletons. So \eqref{eq:fx1} implies that  $f(x_1)=f(x_1')$ only if $\nu_1(x_1)=\nu_1(x_1')$. But then
 $$
 x_1=f(x_1)-\nu_0^{-1}\big(\nu_1(x_1)\big)=f(x_1')-\nu_0^{-1}\big(\nu_1(x_1')\big)=x_1'.
 $$
 So $f$ is one-to-one.
 Finally we show that $f$ is an immersion.
To see this note that
$$
df=d\nu_0^{-1}\circ d\nu_1+I.
$$
Now suppose, towards a contradiction, that $df(v)$ vanishes for some nonzero vector $v$. Then 
$d\nu_0^{-1}\circ d\nu_1$ has a negative eigenvalue. But, recall that $dv_0$ and $d\nu_1$ are self-adjoint operators (this is a basic fact from classical differential geometry).  Furthermore all the eigenvalues of $d\nu_0$ are positive since $\d K_0$ has positive curvature. Consequently, $d\nu_0^{-1}$ is also a self-adjoint operator with positive eigenvalues. Further, since $\d K_1$ is convex, $d\nu_1$ has nonnegative eigenvalues. It follows then that $d\nu_0^{-1}\circ d\nu_1$ may not have any negative eigenvalues, which is the contradiction we seek,  by the following basic fact:

\begin{lem}\label{lem:AB}
Let $A$, $B\colon\R^n\to\R^n$ be self-adjoint linear operators. Suppose that the eigenvalues of $A$ are positive and the eigenvalues of $B$ are nonnegative. Then the eigenvalues of $AB$ are nonnegative.
\end{lem}
\begin{proof}
Since $A$ is self-adjoint and has positive eigenvalues, the associated  quadratic form $Q_A(\cdot):=\l A(\cdot),\cdot\r$ is positive definite. Similarly, since $B$ is self-adjoint and has nonnegative eigenvalues,  $Q_B(\cdot):=\l B(\cdot),\cdot\r$ is nonnegative. Suppose now, towards a contradiction, that $AB(v)=-\lambda v$ for some $v\in\R^n-\{o\}$ and $\lambda>0$. Then $B(v)\neq o$ and consequently 
$$
0<Q_A\big(B(v)\big)=\big\l AB(v),B(v)\big\r=\big\l -\lambda v,B(v)\big\r=-\lambda Q_B(v)\leq 0,
$$
which is the desired contradiction.
\end{proof}

The last lemma completes the proof that $f$ given by \eqref{eq:nu0} is a $C^{k-1}$ diffeomorphism between an open subset of $\d K_1$ and $\d K$. Thus, since $\d K_1$ is $C^k$, it follows that  $\d K$ is (at least) $C^{k-1}$. 
But   the Gauss map $\nu$ of $\d K$ is also $C^{k-1}$, because
$$\nu(x)=\nu_1(x_1)=\nu_1\circ f^{-1}(x).$$
 So it follows that $\d K$ is actually $C^k$ by the following observation:

\begin{lem}\label{lem:gaussregularity}
Let $M\subset\R^n$ be a $C^{k-1}$, $k\geq 2$,  immersed oriented hypersurface, and suppose that the Gauss map $\nu\colon M\to\S^{n-1}$ is also $C^{k-1}$. Then $M$ is actually $C^k$.
\end{lem}
\begin{proof}
Let $p\in M$ and $U$ be a small neighborhood of $p$ in $M$. Further let $e\colon M\times\R\to \R^n$ be the \emph{end point map} given by $e(p,r):=p+r\nu(p)$. It follows from the tubular neighborhood (or the inverse function) theorem that $e\colon U\times (-\epsilon,\epsilon)\to V\subset\R^n$ is a $C^{k-1}$ diffeomorphism, for some $\epsilon>0$ and $U$ sufficiently small. Then the projection map $\pi\colon V\to U$ given by $\pi(x):=Pr_1\circ e^{-1}(x)$, where $Pr_1\colon M\times\R \to M$ is projection into the first component, is well defined and is $C^{k-1}$. Now let $d\colon V\to \R$ be the \emph{signed distance function} from $U$, which is given by 
$d(x):=\l x-\pi(x),\nu(\pi(x))\r$. The gradient of $d$ is then given by $(\grad d)(x)=\nu(\pi(x))$ which is $C^{k-1}$. So $d$ is $C^k$, and, since $d$ is a submersion, it follows that $U=d^{-1}(0)$ is a $C^{k}$ hypersurface.
\end{proof}

\subsection{Analyticity}
If $K_0$ and $K_1$ are analytic, and $K_0$ has positive curvature, then of course all the results of  Section \ref{sec:highreg} still hold. In particular, the parameterization $f$ given by \eqref{eq:nu0} would imply that $K=K_0+K_1$ is analytic as soon as we check that the Gauss maps $\nu_0$ and $\nu_1$ are analytic. But the Gauss map $\nu\colon M\to\S^{n-1}$ of an orientable analytic hypersurface $M\subset\R^n$ is always analytic, i.e., if  $f\colon U\subset\R^{n-1}\to M$  is any analytic local parameterization of $M$, then $\nu\circ f\colon U\to\R^n$ is analytic. To see this note that, for any fixed vector $v_0\in\R^n$, the 
projection of $v_0$ into the tangent space $T_{f(p)}M$ is given by 
$$
\ol v_0(p):=\sum_{i=1}^{n-1}\left\l v_0,  \frac{\d f}{\d x_i}(p)\right\r  \frac{\d f}{\d x_i}(p).
$$
Thus $\ol v_0\colon U\to\R^n$ is analytic. On the other hand, 
 if  we choose $v_0$ so that it is not tangent to $f(U)$ (which is always possible assuming $U$ is small), then
 $$
\nu\circ f(p)= \frac{v_0-\ol v_0(p)}{\|v_0-\ol v_0(p)\|}.
 $$
 So we conclude that $\nu\circ f$ is analytic.

\subsection{Curvature}
Lastly we show that if $K_0$ and $K_1$  have  positive curvature, then $K$  also  has positive curvature. Once again let $\nu_0$, $\nu_1$, and $\nu$ denote the Gauss maps of $\d K_0$, $\d K_1$, and $\d K$ respectively. Then, by Lemma \ref{lem:Fu}, for every $u\in\ol\nc(K)$ we have
$$
\nu^{-1}(u)=F_u(K)=F_u(K_0)+F_u(K_1)=\nu_0^{-1}(u)+\nu_1^{-1}(u).
$$
Note that $\nu_0^{-1}$ and $\nu_1^{-1}$ are $C^1$ by the inverse function theorem and the positive curvature assumption on $K_0$ and $K_1$. Consequently  $\nu^{-1}\colon \ol\nc(K)\to\d K$ is also a well-defined $C^1$ map. In particular we may compute that
$$
d\nu^{-1}_u=d(\nu_0^{-1})_u+d(\nu_1^{-1})_u,
$$
for any $u\in\ol\nc(K)$.
Now let $v_1,\dots, v_n\in T_u\S^{n-1}$ be the eigenvectors of $\nu_0^{-1}$ at $u$. Then $d(\nu_0^{-1})_u(v_i)=v_i/k^0_i(\nu^{-1}(u))$, where $k^0_i$ are the principal curvatures of $\d K_0$, which are all positive by assumption.  So we have
$$
d\nu^{-1}_u(v_i)=\frac{1}{k^0_i(\nu^{-1}(u))}v_i+d(\nu_1^{-1})_u(v_i).
$$
Thus $d\nu^{-1}_u(v_i)=0$ if and only if $v_i$ is an eigenvector of $d(\nu_1^{-1})_u$, with a negative eigenvalue. But the eigenvalues of  $d(\nu_1^{-1})_u$ are reciprocals of the  principal curvatures of $\d K_1$ which are all positive by assumption. Hence $d\nu^{-1}_u(v_i)\neq 0$, and since $v_i$ are linearly independent, it follows that $\nu^{-1}$, and consequently $\nu$, is nonsingular. So $\d (K_0+K_1)$ has nonzero curvature, which since $K_0+K_1$ is convex, must be positive.

\begin{note}\label{note:regularity}
The condition in item (3) of Proposition \ref{prop:sum} that $K_0$ have positive curvature is necessary. Indeed Kiselman \cite{kiselman} has shown that there exist $C^\infty$ convex bodies whose Minkowski sum is not even $C^2$, see also \cite{boman, krantz&parks}.
\end{note}

\section{Proofs of the Main Results}\label{sec:proofs}
Finally we are ready to prove the theorems mentioned in the introduction. Recall that, by the definition of the asymptotic topology on  $\d\K^n$ (Section \ref{subsec:top}), we  only need to construct our deformations $K_t$ for the spaces of convex bodies, for then $\d K_t$ yields the corresponding deformations for the spaces of convex hypersurfaces.
To construct the deformations we seek, we begin by translating our convex bodies until their \emph{apex} passes through the origin as described below.

\subsection{The apex}\label{subsec:apex}
For any $K\in\K^n$, $\apex(K)\subset\d K$ is an affine space which is defined as follows. Let $\d H$ be the support hyperplane of $K$ with outward normal $-\cd(K)$, and set $K':=\d H\cap K$. By Lemma \ref{lem:rc4}, $K'=\ol K'+L'$, where $L'$ is the linearity space of $K'$ and $\ol K'$ is the projection of $K'$ into $L'^\perp$. Now note that, since by Proposition \ref{prop:direction}
$\cd(K)$ is balanced, $\rc(K')$ is not relatively proper; therefore, $K'\not\in\K^{n-1}$, by Lemma \ref{lem:relbd}, where  we have identified $\d H$ with $\R^{n-1}$. Consequently, by Lemma \ref{lem:rc4}, $\ol K'$ is compact, and so its center of mass $\cm(\ol K')$ is well-defined. We may then set
$$
\apex(K):=\cm(\ol K')+L'.
$$
Thus we obtain an affine subset of $\d K$ which ranges from a single point (when $K\in\K^n_+$) to a  hyperplane (when $K$ is a half-space). Further it is not hard to see that $K\mapsto\apex(K)$ is asymptotically continuous, since $K\mapsto\cd(K)$ is asymptotically continuous by Proposition \ref{prop:direction}. Now  for every $K\in \K^n$ let $p$ be the closest point of 
$\apex(K)$ to $o$. Then
$$
K_t:=K-t\,p(K)
 $$
 gives a strong deformation retraction $\K^n\to\ol \K^n$, where $\ol\K^n$ is the space of those bodies $K\in\K^n$ with $o\in\apex(K)$.

\subsection{Proof of Theorem \ref{thm:main}}\label{subsec:proof1}
By the  discussion in Section \ref{subsec:apex} we just need to construct  a strong deformation 
retraction of $\ol \K^n$ into $\H^n$. To this end, let $H^t\subset\R^n$ be the family of hyperboloidal convex bodies given by  
$$
x_n\geq\left(\sqrt{1+\sum_{i=1}^{n-1}x_i^2}-1\right)\frac{1-t}{t}
$$ for $0< t\leq1$, and set $H^0$ equal to the upper-half of the $x_n$ axis. Note that $H^t$ gives an asymptotically continuous family of convex sets which range from the half-line $H^0$ to the upper-half plane $H^1=\mathbf{H}^n$, see Figure \ref{fig:hyperbolas}.
\begin{figure}[h]
\begin{center}
\begin{overpic}[width=5.2in ]{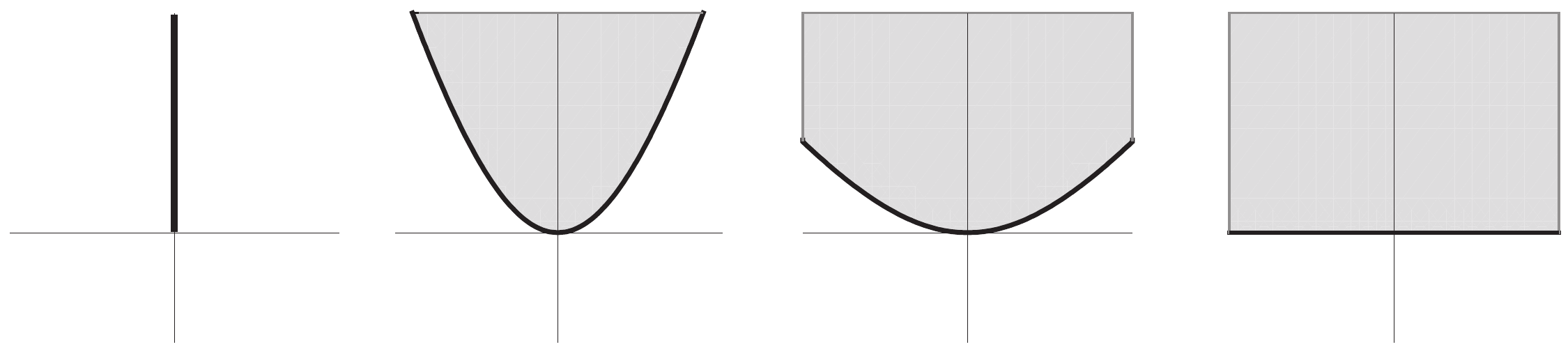}
\put(15,2){$H^0$}
\put(40,2){$H^\frac{1}{4}$}
\put(66,2){$H^\frac{1}{2}$}
\put(94,2){$H^1$}
\end{overpic}
\caption{}
\label{fig:hyperbolas}
\end{center}
\end{figure}
Next let $H^t_u$ be the object which is obtained by a rotation of $H^t$ about $o$ so that its central  direction coincides with $u$.
 Then, for $K\in\K^n$, we define
\begin{equation}\label{eq:main1}
K_t:=K+H^t_{\cd(K)}.
\end{equation}
This gives the desired retraction of $\ol \K^n$ into $\H^n$. In particular note that $K_0=K$ since $H^0_{\cd(K)}\subset \rc(K)$ (if $\ell\subset \rc(K)$ is any half-line, then $K+\ell=K$). Further $K_1=H^1_{\cd(K)}$, since $K\subset H^1_{\cd(K)}$. The asymptotic continuity of $K_t$, for $0<t\leq 1$, and that $K_t\in\K^n$, follows from Proposition \ref{prop:continuous2}, since $\cd(H^t_{\cd(K)})=\cd(K)$ (note that in applying Proposition \ref{prop:continuous2} here we are also implicitly using Proposition \ref{prop:direction} which guarantees that the central directions are always balanced).  Further it is clear that $K_t\overset{a}{\to}K_0$ as $t\to 0$. Thus $K_t$ is asymptotically continuous. Furthermore, the regularity preserving properties of $K_t$ follow from Proposition \ref{prop:sum}. Finally, we have to check that $\ol\nc(K_t)$ continuously and monotonically shrinks to a point. That $\ol\nc(K_t)$ changes continuously follows from Proposition \ref{lem:olnc}. Further $\ol\nc(K_1)=-\cd(K)$, a single point. It remains then to check the monotonicity, i.e., to show that 
$d_h(\ol\nc(K_t),-\cd(K))\to 0$ monotonically. To see this recall that by Lemma \ref{lem:Fu}
$$
\ol\nc(K_t)=\ol\nc(K)\cap\ol\nc\big(H^t_{\cd(K)}\big),
$$
and note that $d_h(\ol\nc(H^t_{\cd(K)}),-\cd(K))\to 0$ monotonically.

 \subsection{Proof of Theorem \ref{thm:main2}}\label{subsec:proof2}
Similar to the proof of  Theorem \ref{thm:main}, we may confine our attention to the space $\ol\K^n_+:=\ol \K^n\cap \K^n_+$.
Now, for $u\in\S^{n-1}$, let $P_u$ be the paraboloidal convex body which is obtained by rotating the solid paraboloid given by $x_n\geq\sum_{i=1}^{n-1}x_i^2$ about $o$ until $u$ becomes its central direction. Further, let $S_{\lambda,u}\colon\R^n\to\R^n$ denote the stretching along the direction $u$ by the factor $\lambda$, i.e., 
$$S_{\lambda,u}(x):=x+(\lambda-1)\l x,u\r u.$$ Then, for any $K\in \ol \K^n_+$ set 
\begin{equation}\label{eq:main2}
K_t:=(1-t)S_{\frac{1}{1-t},\cd(K)}(K)+t P_{\cd(K)}.
\end{equation}
It is obvious that $K_0=K$ and $K_1=P_{\cd(K)}$. Further the continuity of $K_t$ for $0<t<1$, and that $K_t\in\K^n_+$, follows immediately from Proposition \ref{prop:continuous} (which again applies via Proposition \ref{prop:direction}). Furthermore, note that as $t\to 1$, $(1-t)S_{\frac{1}{1-t},\cd(K)}(K)$ converges asymptotically to the half-line generated by $\cd(K)$ which lies in $P_{\cd(K)}$. Thus $K_t\overset{a}{\to}K_1$ as $t\to 1$. Similarly, since $t P_{\cd(K)}$ converges asymptotically to the half-line generated by $\cd(K)$, we have $K_t\overset{a}{\to}K_0$ as $t\to 0$. So we conclude that $K_t$ is asymptotically continuous, which shows that the total curvature $t\mapsto\tau(K_t)$ is continuous as well by Proposition \ref{lem:olnc}. Finally, the regularity and curvature preserving properties of $K_t$ again follow from Proposition \ref{prop:sum}.
  
 \subsection{Other topological types}\label{subsec:proof3}
 Here we address the case of unbounded  convex bodies $K\varsubsetneq \R^n$ whose boundary is not homeomorphic to $\R^{n-1}$. In that case, it follows from Lemma \ref{lem:rc4}  that $\d K$ is homeomorphic to $\S^{n-m-1}\times\R^m$, $m=1,\dots,n-1$. Thus there are, in addition to the case of $\K^n$, $n-1$ other topological types of proper unbounded convex bodies in $\R^n$, which we denote respectively by $\K^{n,m}$. Let $\B^{n,m}\subset\K^{n,m}$ be the subspace which is obtained by the action of $SO(n)$ on $\B^{n-m}\times\R^m\subset\R^n$. Then,   the contractibility of the space of compact convex bodies quickly  yields  that

\begin{thm}\label{thm:main3}
$\K^{n,m}$  (resp. $\d \K^{n,m}$) admits a  regularity preserving strong deformation retraction onto $\B^{n,m}$ (resp. $\d \B^{n,m}$) with respect to the asymptotic topology.
\end{thm}
\begin{proof}
  If $K\in\K^{n,m}$, then  recall that $\rc(K)$ contains a nontrivial maximal linear subspace $L$ of dimension $m$ (see Section \ref{subsec:rc}). Let $\K^{n,m}_L$ be the collection of all bodies in $\K^{n,m}$ with linearity space $L$. Next  let  $\pi\colon\R^n\to L^\perp$ be the orthogonal projection. Then $\pi(K)$ is a compact convex body in $L^\perp$ for all $K\in \K^{n,m}_L$. Consequently, there is a homotopy $\ol K_t:= (1-t)\pi(K)+t B_L$ which deforms $\pi(K)$ to the unit ball $B_L:=\mathbf{B}^n\cap L$ centered at the origin of $L^\perp$. Then $K_t:=\pi^{-1}(\ol K_t)$ gives a homotopy between $K$ and $\pi^{-1}(B_L)=B_L + L$. Further, since $\ol K_t$ is regularity preserving (e.g., by Proposition \ref{prop:sum}), it follows that $K_t$ is regularity preserving as well.
  \end{proof}

\section*{Acknowledgement}
The author thanks Harold Rosenberg for his interesting question on deformations of complete positively curved hypersurfaces  \cite{rosenberg:question}, which was the prime stimulus for this work. Thanks also to Gerald Beer and Roger Wets for informative communications with regard to various hyperspace topologies.

\bibliographystyle{abbrv}

\end{document}